\documentclass{amsart}

\usepackage{amsfonts,amsmath,amssymb,amsthm,amscd,amsxtra}
\usepackage{enumerate,verbatim}

\usepackage[all,2cell,ps]{xy}
\usepackage[pagebackref]{hyperref}

\theoremstyle{plain} 

\newtheorem{thm}{Theorem}[section]

\newtheorem{prop}[thm]{Proposition}
\newtheorem{cor}[thm]{Corollary}

\theoremstyle{definition}
\newtheorem{lem}[thm]{Lemma}
\newtheorem{dfn}[thm]{Definition}

\newtheorem{eg}[thm]{Example}

\newtheorem{ques}[thm]{Question}

\numberwithin{equation}{section}
\newcommand{\fm}{\mathfrak{m}}
\newcommand{\fn}{\mathfrak{n}}
\newcommand{\fp}{\mathfrak{p}}

\newcommand{\bR}{\mathbb{R}}

\newcommand{\bQ}{\mathbb{Q}}

\newcommand{\bZ}{\mathbb{Z}}

\newcommand{\D}{\mathbb{D}}

\newcommand{\NN}{\mathbb{N}}

\DeclareMathOperator{\Ass}{Ass}

\DeclareMathOperator{\CI}{\textnormal{CI-dim}}
\DeclareMathOperator{\G-dim}{\textnormal{G-dim}}
 \DeclareMathOperator{\cx}{cx}
\DeclareMathOperator{\Tr}{Tr}

\DeclareMathOperator{\coker}{coker}
\DeclareMathOperator{\depth}{depth}

\DeclareMathOperator{\Ext}{Ext}
\DeclareMathOperator{\Hom}{Hom}

\DeclareMathOperator{\len}{length}
\DeclareMathOperator{\pd}{pd}

\DeclareMathOperator{\Spec}{Spec}

\DeclareMathOperator{\Tor}{Tor}
\DeclareMathOperator{\cod}{codim}
\DeclareMathOperator{\emb}{embdim}

\def\urltilda{\kern -.15em\lower .7ex\hbox{\~{}}\kern .04em}
\def\urldot{\kern -.10em.\kern -.10em}\def\urlhttp{http\kern -.10em\lower -.1ex
\hbox{:}\kern -.12em\lower 0ex\hbox{/}\kern -.18em\lower 0ex\hbox{/}}

\begin{document}

\title[vanishing of cohomology over complete intersection rings]{vanishing of cohomology over complete intersection rings}
\author[A. Sadeghi]{Arash Sadeghi}
\address{School of Mathematics, Institute for Research in Fundamental Sciences (IPM), P.O. Box: 19395-5746, Tehran, Iran.}
\email{sadeghiarash61@gmail.com}

\subjclass[2000]{13D07, 13H10}

\keywords{Complete intersection dimension, Complexity, Vanishing of cohomology, Grothendieck group.}

\date{March 18, 2013.}

\begin{abstract}
Let $R$ be a complete intersection ring and let $M$ and $N$ be $R$--modules. It is shown that the vanishing of $\Ext^i_R(M,N)$ for a certain number of consecutive values of $i$ starting
at $n$ forces the complete intersection dimension of $M$ to be at most $n-1$. We also estimate the complete intersection dimension of $M^*$, the dual of $M$, in terms of vanishing of the cohomology modules, $\Ext^i_R(M,N)$.
\end{abstract}

\maketitle{}

\setcounter{tocdepth}{1}
\tableofcontents
\section{Introduction}
In this paper, we study the relationship between the vanishing of $\Ext^i_R(M,N)$
for various consecutive values of $i$, and the complete intersection dimensions of $M$ and $M^*$, the dual of $M$.
The vanishing of homology was first studied by Auslander \cite{A2}. For two finitely generated modules $M$ and $N$ over an unramified regular local ring $R$, he proved that if $\Tor_i^R(M,N)=0$ for some $i>0$, then $\Tor_n^R(M,N)=0$ for all $i\geq n$.
In \cite{L}, Lichtenbaum settled the ramified case.
It is easy to see that a similar statement is not true in general, with Tor replaced by Ext.
In \cite{Jot2}, Jothilingam studied the vanishing of cohomology by using the rigidity Theorem of Auslander.
For two nonzero modules $M$ and $N$ over a regular local ring $R$, he proved that if $M$ satisfies $(S_n)$ for some $n\geq0$ and
$\Ext^i_R(M,N)=0$ for some positive integer $i$ such that $i\geq\depth_R(N)-n$, then $\Ext^j_R(M,N)=0$ for all $j\geq i$.
In \cite{JoDu}, Jothilingam and Duraivel studied the relationship between the vanishing of $\Ext^i_R(M,N)$
and the freeness of $M^*$. For two nonzero modules $M$ and $N$ over a regular local ring $R$, they proved that
if $\Ext^i_R(M,N)=0$ for all $1\leq i\leq\max\{1, \depth_R(N)-2\}$, then $M^*$ is free.
In this paper we are going to generalize these results.

An $R$--module $M$ is said to be $\emph{c-rigid}$ if for all $R$--modules $N$, $\Tor_{i+1}^R(M,N)=\Tor_{i+2}^R(M,N)=\cdots=\Tor_{i+c}^R(M,N)=0$ for some $i\geq0$ implies that $\Tor_n^R(M,N)=0$ for all $n>i$. If $c=1$ then we simply say that $M$ is rigid.

The aim of this paper is to study the following question.
\begin{ques}\label{q}
Let $R$ be a Gorenstein local ring and let $M$ and $N$ be $R$--modules such that $N$ has reducible complexity.
Assume that $n\geq0$, $c>0$ are integers and that $N$ is $c$-rigid. If $\Ext^i_R(M,N)=0$ for all $i$, $1\leq i\leq\max\{c, \depth_R(N)-n\}$,
then what can we say about the Gorenstein dimensions of $M$ and $M^*$?
\end{ques}
In section 2, we collect necessary notations, definitions and some known results which will be used in this paper.

In section 3, we study the Question \ref{q} for rigid modules.
Over a Gorenstein local ring $R$, given nonzero $R$-modules $M$ and $N$ such that $N$ has reducible complexity, we show that if $N$ is rigid and $\Ext^i_R(M,N)=0$ for all $i$, $1\leq i\leq\max\{1,\depth_R(N)-n\}$ and some $n\geq2$, then $\G-dim_R(M^*)\leq n-2$, which is a generalization of \cite[Theorem 1]{JoDu}. In particular, if $M$ satisfies $(S_n)$, then $\G-dim_R(M)=0$ (see Theorem \ref{t12}).
As a consequence, for two nonzero modules $M$ and $N$ over a complete intersection ring $R$, it is shown that if $N$ is rigid and $\Ext^i_R(M,N)=0$ for some positive integer $i\geq\depth_R(N)$, then $\CI_R(M)=\sup\{i\mid\Ext^i_R(M,N)\neq0\}<i$ (see Theorem \ref{cor7}).

In section 4, we generalize \cite[Corollary 1]{Jot2} for modules over a complete intersection ring. For two modules $M$ and $N$ over a complete intersection ring $R$ with codimension $c$, it is shown that if $M$ satisfies $(S_t)$ for some $t\geq0$, $\Ext^i_R(M,N)=0$ for all $i$, $n\leq i\leq n+c$ and some $n>0$ and $\depth_R(N)\leq n+c+t$, then $\CI_R(M)=\sup\{i\mid\Ext^i_R(M,N)\neq0\}<n$ (see Corollary \ref{cor11}).

\section{Preliminaries}
Throughout the paper, $(R, \fm)$ is a commutative Noetherian local ring and all modules are finite (i.e. finitely generated)
$R$--modules. The codimension of $R$ is defined to be the non-negative integer $\emb(R)-\dim(R)$ where $\emb(R)$, the embedding dimension of $R$, is the minimal number of generators of $\fm$. Recall that $R$ is said
to be a complete intersection if the $\fm$-adic completion $\widehat{R}$ of $R$ has the form $Q/(f)$, where $f$ is a regular sequence of $Q$ and $Q$ is a regular local ring. A complete intersection of codimension one is called a hypersurface.
A local ring $R$ is said to be an admissible complete intersection if the $\fm$-adic completion $\widehat{R}$ of $R$ has the form $Q/(f)$, where $f$ is a regular sequence of $Q$ and $Q$ is a power series ring over a field or a discrete valuation ring.
Let $$\cdots\rightarrow F_{n+1} \rightarrow F_{n}\rightarrow F_{n-1}\rightarrow\cdots\rightarrow F_0\rightarrow M\rightarrow0$$ be the minimal free resolution of $M$. Recall that the $n^{\text{th}}$ syzygy of an $R$--module $M$ is the cokernel of the $F_{n+1}\rightarrow F_{n}$ and denoted by $\Omega^nM$, and it is unique up to isomorphism. The $n^{\text{th}}$ Betti number, denoted $\beta_n^R(M)$, is the rank of the free $R$--module $F_n$.
The complexity of $M$ is defined as follows.
$$\cx_R(M)=\inf\{i\in \NN \cup 0\mid \exists \gamma\in \bR  \text{ such that } \beta_n^R(M)\leq\gamma n^{i-1}\text{ for } n\gg0\}.$$
Note that $\cx_R(M)=\cx_R(\Omega^iM)$ for every $i\geq0$. It follows from the definition that
$\cx_R(M)=0$ if and only if $\pd_R(M)<\infty$.
If $R$ is a complete intersection, then the complexity of $M$ is less than or equal to the codimension of $R$ (see \cite{G}).
The complete intersection dimension was introduced by Avramov, Gasharov and Peeva \cite{AGP}. A module of finite complete intersection dimension behaves homologically like a module over a complete intersection.
Recall that a quasi-deformation of $R$ is a diagram $R\rightarrow A\twoheadleftarrow Q$ of local homomorphisms, in which
$R\rightarrow A$ is faithfully flat, and $A\twoheadleftarrow Q$ is surjective with kernel generated by a regular sequence.
The module $M$ has finite complete intersection dimension if there exists such a quasi-deformation for which $\pd_Q(M\otimes_RA)$ is finite.
The complete intersection dimension of $M$, denoted $\CI_R(M)$, is defined as follows.
$$\CI_R(M)=\inf\{\pd_Q(M\otimes_RA)-\pd_Q(A)\mid R\rightarrow A\twoheadleftarrow Q \text{ is a quasi-deformation }\}.$$
The complete intersection dimension of $M$ is bounded above by the projective dimension,
$\pd_R(M)$, of $M$ and if $\pd_R(M)<\infty$, then the equality holds (see \cite[Theorem 1.4]{AGP}).
Every module of finite complete intersection dimension has finite complexity (see \cite[Theorem 5.3]{AGP}).

The concept of modules with reducible complexity was introduced by Bergh \cite{B2}.\\
Let $M$ and $N$ be $R$--modules and consider a homogeneous element $\eta$ in the graded $R$--module
$\Ext^*_R(M,N)=\bigoplus^{\infty}_{i=0}\Ext^i_R(M,N)$. Choose a map $f_{\eta}:\Omega^{|\eta|}_R(M)\rightarrow N$ representing $\eta$, and denote by $K_{\eta}$ the pushout of this map and the inclusion $\Omega^{|\eta|}_R(M)\hookrightarrow F_{|\eta|-1}$.
Therefore we obtain a commutative diagram
$$\begin{CD}
&&&&&&&&\\
  \ \ &&&&  0@>>>\Omega^{|\eta|}M @>>> F_{|\eta|-1}@>>>\Omega^{|\eta|-1}M @>>>0&  \\
                                &&&&&& @VV{f_{\eta}}V @VV V @VV{\parallel} V\\
  \ \  &&&& 0@>>> N @>>> K_{\eta} @>>>\Omega^{|\eta|-1}M@>>>0.&\\
\end{CD}$$\\
with exact rows. Note that the module $K_{\eta}$ is independent, up to isomorphism, of the map $f_{\eta}$ chosen to represent ${\eta}$.
\begin{dfn}
The full subcategory of $R$-modules consisting of the modules having
reducible complexity is defined inductively as follows:
\begin{itemize}
        \item[(i)] Every $R$-module of finite projective dimension has reducible complexity.
         \item[(ii)] An $R$-module $M$ of finite positive complexity has reducible complexity if
                    there exists a homogeneous element $\eta\in\Ext^{*}_R(M,M)$, of positive degree,
                         such that $\cx_R(K_{\eta}) < \cx_R(M)$, $\depth_R(M)=\depth_R(K_{\eta})$ and $K_{\eta}$ has reducible complexity.
\end{itemize}
\end{dfn}
By \cite[Proposition 2.2(i)]{B2}, every module of finite complete intersection dimension has reducible complexity. In particular, every module over a local complete intersection ring has reducible complexity. On the other hand, there are modules having reducible complexity but whose complete intersection dimension is infinite (see for example, \cite[Corollarry 4.7]{BJ}).

The notion of the Gorenstein(or G-) dimension was introduced by Auslander \cite{A1}, and developed by Auslander and Bridger in \cite{AB}.
\begin{dfn}
An $R$--module $M$ is said to be of $G$-dimension zero whenever
\begin{itemize}
            \item[(i)]{\emph{the biduality map $M\rightarrow M^{**}$ is an isomorphism.}}
            \item[(ii)]{\emph{$\Ext^i_R(M,R)=0$ for all $i>0$.}}
            \item[(iii)]{\emph{$\Ext^i_R(M^*,R)=0$ for all $i>0$.}}
\end{itemize}
\end{dfn}
The Gorenstein dimension of $M$, denoted $\G-dim_R(M)$, is defined to be the infimum of all
nonnegative integers $n$, such that there exists an exact sequence
$$0\rightarrow G_n\rightarrow\cdots\rightarrow G_0\rightarrow  M \rightarrow 0$$
in which all the $G_i$ have $G$-dimension zero.
By \cite[Theorem 4.13]{AB}, if $M$ has finite Gorenstein dimension, then $\G-dim_R(M)=\depth R-\depth_R(M)$.
By \cite[Theorem 1.4]{AGP},
$\G-dim_R(M)$ is bounded above by the complete intersection dimension, $\CI_R(M)$, of $M$ and if $\CI_R(M)<\infty$, then the equality holds.

Let $R$ be a local ring and let $M$ and $N$ be finite nonzero $R$-modules. We say the pair $(M,N)$ satisfies
the depth formula provided:
$$\depth_R(M\otimes_RN)+\depth R=\depth_R(M)+\depth_R(N).$$
The depth formula was first studied by Auslander \cite{A2} for finite modules of finite projective dimension.
In \cite{HW}, Huneke and Wiegand proved that the depth formula holds for $M$ and $N$ over complete intersection rings $R$ provided
$\Tor_i^R(M,N)=0$ for all $i>0$. In \cite{BJ}, Bergh and Jorgensen generalize this result for modules with reducible complexity over a local Gorenstein ring. More precisely, they proved the following result:
\begin{thm}\label{th}\cite[Corollary 3.4]{BJ}
Let $R$ be a Gorenstein local ring and let $M$ and $N$ be nonzero $R$--modules. If $M$ has reducible complexity and $\Tor_i^R(M,N)=0$ for all $i>0$, then $\depth_R(M\otimes_RN)+\depth R=\depth_R(M)+\depth_R(N)$.
\end{thm}

We denote by $G(R)$ the Grothendieck group of finite modules over $R$, that is, the
quotient of the free abelian group of all isomorphism classes of finite $R$--modules
by the subgroup generated by the relations coming from short exact sequences of finite $R$-modules. We also denote by $\overline{G}(R)= G(R)/[R]$, the reduced Grothendieck group. For an abelian group $G$, we
set $G_{\bQ}=G\otimes_{\bZ}\bQ$.

Let $P_1\overset{f}{\rightarrow}P_0\rightarrow M\rightarrow 0$ be a
finite projective presentation of $M$. The transpose of $M$, $\Tr M$,
is defined to be $\coker f^*$, where $(-)^* := \Hom_R(-,R)$, which satisfies in the exact sequence
\begin{equation} \label{n1}
0\rightarrow M^*\rightarrow P_0^*\rightarrow P_1^*\rightarrow \Tr M\rightarrow 0
\end{equation}
and is unique up to projective equivalence. Thus the minimal projective
presentations of $M$ represent isomorphic transposes of $M$.
Two modules $M$ and $N$ are called \emph{stably isomorphic} and write $M\approx N$ if $M\oplus P\cong N\oplus Q$ for some projective
modules $P$ and $Q$. Note that $M^*\approx\Omega^2\Tr M$ by the exact sequence (\ref{n1}).

The composed functors $\mathcal{T}_k:=\Tr\Omega^{k-1}$ for $k>0$  introduced by Auslander and Bridger in \cite{AB}. If $\Ext^i_R(M,R)=0$ for some $i>0$, then it is easy to see that $\mathcal{T}_iM\approx\Omega\mathcal{T}_{i+1}M$.

We frequently use the following Theorem of Auslander and Bridger.
\begin{thm}\cite[Theorem 2.8]{AB}\label{a1}
Let $M$ be an $R$--module and $n\geq0$ an integer. Then there are exact sequences of functors:
\begin{equation}\tag{\ref{a1}.1}
0\rightarrow\Ext^1_R(\mathcal{T}_{n+1}M,-)\rightarrow\Tor_n^R(M,-)\rightarrow\Hom_R(\Ext^n_R(M,R),-)
\rightarrow\Ext^2_R(\mathcal{T}_{n+1}M,-),
\end{equation}
\begin{equation}\tag{\ref{a1}.2}
\Tor_2^R(\mathcal{T}_{n+1}M,-)\rightarrow(\Ext^n_R(M,R)\otimes_R-)\rightarrow\Ext^n_R(M,-)\rightarrow
\Tor_1^R(\mathcal{T}_{n+1}M,-)\rightarrow0.
\end{equation}
\end{thm}
For an integer $n\geq0$, we say $M$ satisfies $(S_n)$ if $\depth_{R_{\fp}}(M_{\fp})\geq\min\{n, \dim(R_\fp)\}$
for all $\fp\in\Spec(R)$. If $R$ is Gorenstein, then $M$ satisfies $(S_n)$ if and only if $\Ext^i_R(\Tr M,R)=0$ for all $1\leq i\leq n$
(see \cite[Theorem 4.25]{AB}). In particular, $M$ satisfies $(S_2)$ if and only if it
is reflexive, i.e., the natural map $M\rightarrow M^{**}$ is bijective, where $M^* = \Hom_R(M,R)$ (see \cite[Theorem 3.6]{EG}).

The following results will be used throughout the paper.
\begin{thm}\label{th1}
Let $R$ be a local complete intersection ring and let $M$ and $N$ be $R$--modules. Then $\Tor_i^R(M,N)=0$ for all $i\gg0$ if and only if $\Ext^i_R(M,N)=0$ for all $i\gg0$. Moreover, if $R$ is a hypersurface, then either $\pd_R(M)<\infty$ or $\pd_R(N)<\infty$.
\end{thm}
\begin{proof}
See \cite[Theorem 6.1]{AvBu} and \cite[Proposition 5.12]{AvBu}.
\end{proof}
\begin{thm}\label{th2}
Let $R$ be a local ring and let $M$ and $N$ be nonzero $R$--modules. If $\Ext^i_R(M,N)=0$ for all $i\gg0$ and $\G-dim_R(M)<\infty$, then
the following statements hold true.
\begin{enumerate}[(i)]
\item{If $\CI_R(M)<\infty$, then $\CI_R(M)=\sup\{i\mid\Ext^i_R(M,N)\neq0\}$.}
\item{If $\CI_R(N)<\infty$, then $\G-dim_R(M)=\sup\{i\mid\Ext^i_R(M,N)\neq0\}$}.
\end{enumerate}
\end{thm}
\begin{proof}
See \cite[Theorem 4.2]{AY} and \cite[Theorem 4.4]{As}.
\end{proof}
\begin{thm}\label{th3}
Let $R$ be a local ring, and $M$, $N$ two $R$--modules. If $\CI_R(M)=0$, then $\Ext^i_R(M,N)=0$ for all $i>0$ if and only if
$\Tor_i^R(\Tr M,N)=0$ for all $i>0$.
\end{thm}
\begin{proof}
First note that $\CI_R(\Tr M)=0$ by \cite[Lemma 3.3]{As} and $M\approx\Tr\Tr M$. Now the assertion is clear by \cite[Proposition 3.4]{As}.
\end{proof}

\section{Vanishing of Ext for rigid modules}
We start this section by estimate the Gorenstein dimension of the transpose of $M$ in terms of vanishing of the cohomology modules, $\Ext^i_R(M,N)$.
\begin{lem}\label{le}
Let $R$ be a Gorenstein ring and let $M$ and $N$ be nonzero $R$--modules. Assume that $n\geq0$ is an integer and that the following conditions hold.
\begin{enumerate}
\item{$\Ext^i_R(M,N)=0$ for all $1\leq i\leq\max\{1,\depth_R(N)-n\}$}.
\item{$N$ is rigid}.
\item{$N$ has reducible complexity}.
\end{enumerate}
Then $\G-dim_R(\Tr M)\leq n$ and $\Tor_i^R(\Tr M,N)=0$ for all $i>0$.
\end{lem}
\begin{proof}
If $\Tr M=0$, then $\G-dim_R(\Tr M)=0$ and we have nothing to prove so let $\Tr M\neq0$. As $\Ext^1_R(M,N)=0$, $\Tor_1^R(\mathcal{T}_2M,N)=0$ by the exact sequence (\ref{a1}.2). Since $N$ is rigid, we have $\Tor_i^R(\mathcal{T}_2M,N)=0$ for all $i>0$.
It follows from the exact sequence (\ref{a1}.2) again that $\Ext^1_R(M,R)\otimes_RN=0$ and since $N$ is nonzero, $\Ext^1_R(M,R)=0$. Now it is easy to see that $\mathcal{T}_1M\approx\Omega\mathcal{T}_2M$ and so $\Tor_i^R(\Tr M,N)=0$ for all $i>0$. Therefore we have the following equality.
\begin{equation}\tag{\ref{le}.1}
\depth_R(\Tr M\otimes_RN)+\depth R=\depth_R(\Tr M)+\depth_R(N),
\end{equation}
by Theorem \ref{th}.
Set $t=\depth_R(N)-n$. We argue by induction on $t$. If $t\leq 1$, then $\depth_R(N)\leq n+1$.
If $\depth_R(N)=0$, then it is clear that $\depth_R(\Tr M)=\depth R$ by (\ref{le}.1) and so $\G-dim_R(\Tr M)=0$ by Auslander-Bridger formula. Now let $0<\depth_R(N)\leq n+1$. As $M\approx\Tr\Tr M$, we obtain the following exact sequence
$$0\rightarrow\Ext^1_R(M,N)\rightarrow\Tr M\otimes_RN\rightarrow\Hom_R((\Tr M)^*,N)\rightarrow\Ext^2_R(M,N),$$
from the exact sequence (\ref{a1}.1). As $\Ext^1_R(M,N)=0$, we get the following exact sequence.
\begin{equation}\tag{\ref{le}.2}
0\rightarrow\Tr M\otimes_RN\rightarrow\Hom_R((\Tr M)^*,N)\rightarrow\Ext^2_R(M,N).
\end{equation}
Therefore $\Ass_R(\Tr M\otimes_RN)\subseteq\Ass_R(\Hom_R((\Tr M)^*,N))\subseteq\Ass_R(N)$, by the exact sequence (\ref{le}.2).
Hence $\depth_R(\Tr M\otimes_RN)>0$. Now by (\ref{le}.1), it is easy to see that $\depth_R(\Tr M)\geq \depth R-n$ and so
$\G-dim_R(\Tr M)\leq n$.

Now suppose that $t>1$ and consider the following exact sequence
\begin{equation}\tag{\ref{le}.3}
0\rightarrow\Omega M\rightarrow F\rightarrow M\rightarrow0,
\end{equation}
where $F$ is a free $R$--module.
From the exact sequence (\ref{le}.3), we obtain the following exact sequence
$$0\rightarrow M^*\rightarrow F^*\rightarrow (\Omega M)^*\rightarrow\D(M)\rightarrow\D(F)\rightarrow\D(\Omega M)\rightarrow0.$$
Where $\D(X)\approx\Tr X$ for all $R$--modules $X$ by \cite[Lemma 3.9]{AB}.
As $\Ext^1_R(M,R)=0$, we get the following exact sequence
\begin{equation}\tag{\ref{le}.4}
0\rightarrow \D(M)\rightarrow \D(F)\rightarrow \D(\Omega M)\rightarrow0.
\end{equation}
Note that $\D(F)$ is free.
As $\Ext^i_R(\Omega M,N)\cong\Ext^{i+1}_R(M,N)=0$ for all $1\leq i\leq\depth_R(N)-n-1$, we have $\G-dim_R(\Tr\Omega M)\leq n+1$ by induction hypothesis. Therefore, $\G-dim_R(\Tr M)\leq n$ by the exact sequence (\ref{le}.4).
\end{proof}
\begin{thm}\label{t12}
Let $R$ be a Gorenstein ring and let $M$ and $N$ be nonzero $R$--modules such that $N$ has reducible complexity. Assume that $N$ is rigid and that $n\geq0$ is an integer. Then the following statements hold true.
\begin{enumerate}[(i)]
\item{If $\Ext^i_R(M,N)=0$ for all $1\leq i\leq\max\{1,\depth_R(N)-n\}$
and $M$ satisfies $(S_n)$, then $\G-dim_R(M)=0$.}
\item{If $\Ext^i_R(M,N)=0$ for all $1\leq i\leq\max\{1,\depth_R(N)-n\}$,
then $\G-dim_R(M^*)\leq n-2$.}
\end{enumerate}
\end{thm}
\begin{proof}
(i). First note that $\G-dim_R(\Tr M)=\sup\{i\mid\Ext^i_R(\Tr M,R)\neq0\}$ by \cite[Theorem 4.13]{AB}. As $M$ satisfies $(S_n)$, $\Ext^i_R(\Tr M,R)=0$ for all $1\leq i\leq n$ by \cite[Theorem 4.25]{AB}. On the other hand, $\G-dim_R(\Tr M)\leq n$ by Lemma \ref{le}. Therefore $\G-dim_R(\Tr M)=0$ and so $\G-dim_R(M)=0$ by \cite[Lemmm 4.9]{AB}.

(ii). Note that $M^*\approx\Omega^2\Tr M$. By Lemma \ref{le}, $\G-dim_R(\Tr M)\leq n$ and so $\G-dim_R(M^*)\leq n-2$.
\end{proof}
The following is a generalization of \cite[Theorem 1]{JoDu}.
\begin{cor}\label{co}
Let $R$ be a complete intersection and let $M$ and $N$ be nonzero $R$--modules. Assume that $N$ is a rigid module of maximal complexity. If
$\Ext^i_R(M,N)=0$ for all $i$, $1\leq i\leq\max\{1,\depth_R(N)-2\}$, then $M^*$ is free.
\end{cor}
\begin{proof}
By Lemma \ref{le}, $\Tor_i^R(\Tr M,N)=0$ for all $i>0$. As $M^*\approx\Omega^2\Tr M$,
$\Tor_i^R(M^*,N)=0$ for all $i>0$ and so $\cx_R(M^*)+\cx_R(N)\leq\cod R$ by \cite[Theorem II]{AvBu}. Since $N$ has maximal complexity, it follows that $\cx_R(M^*)=0$. Therefore, $\pd_R(M^*)=\G-dim_R(M^*)=0$ by Theorem \ref{t12}(ii).
\end{proof}
It is well-known that over a regular local ring every finite module is rigid. In the following we collect some other examples of rigid modules.
\begin{eg}
\begin{enumerate}[(i)]
\item{A class of rigid modules was discovered by Peskine and Szpiro \cite{PS}. They proved that if $R$ is local, and the minimal free resolution of $M$ over $R$ is of the form
$$0\rightarrow R^m\rightarrow R^{k+m}\rightarrow R^k\rightarrow0,$$
for some $m>0$ and $k>0$, then $M$ is rigid. In \cite{T}, Tchernev discovered a new class of rigid modules. He showed that if $R$ is local, and the minimal free resolution of $M$ over $R$ is of the form
$$0\rightarrow R^k\rightarrow R^{m+1}\rightarrow R^m\rightarrow0,$$
for some $m>0$ and $k>0$, then $M$ is rigid (\cite[Theorem 3.6]{T}).}

\item{ Let $R$ be an admissible hypersurface with isolated singularity and let $N$ be an $R$--module. If $[N]=0$ in $\overline{G}(R)_{\bQ}$, then $N$ is rigid \cite[Corollary 4.2]{Da}.}

\item{ Let $(R,\fm)$ be a local hypersurface ring such that $\widehat{R} = S/(f)$ where $(S, \fn)$ is a
complete unramified regular local ring and $f$ is a regular element of $S$ contained in $n^2$. Let $M$ be an $R$-module of finite projective dimension. Then $M$ is rigid \cite[Theorem 3]{L}.}
\end{enumerate}
\end{eg}

In the following, we generalize \cite[Corollary 1]{Jot2}.
\begin{thm}\label{cor7}
Let $R$ be a local complete intersection ring and let $M$ and $N$ be nonzero $R$--modules. Assume the following conditions hold.
\begin{enumerate}[(i)]
\item{$N$ is rigid.}
\item{$M$ satisfies $(S_n)$ for some $n\geq0$.}
\item{$\Ext^i_R(M,N)=0$ for some positive integer $i$ such that $i\geq\depth_R(N)-n$.}
\end{enumerate}
Then $\CI_R(M)=\sup\{j\mid\Ext^j_R(M,N)\neq0\}<i$.
\end{thm}
\begin{proof}
Set $L=\Omega^{i-1}M$. Note that $L$ satisfies $(S_{n+i-1})$ and $\Ext^1_R(L,N)=0$. Now by Theorem \ref{t12}(i), $\CI_R(L)=\G-dim_R(L)=0$.
By Lemma \ref{le}, $\Tor_j^R(\Tr L,N)=0$ for all $j>0$ and so $\Ext^j_R(L,N)=0$ for all $j>0$ by Theorem \ref{th3}. Therefore $\Ext^j_R(M,N)=0$ for all $j\geq i$ and so $\CI_R(M)=\sup\{j\mid\Ext^j_R(M,N)\neq0\}<i$ by Theorem \ref{th2}.
\end{proof}
The following is a generalization of \cite[Corollary 2]{Jot2}
\begin{thm}\label{cor10}
Let $R$ be a local complete intersection ring and let $M$ and $N$ be nonzero $R$--modules. Suppose that $N$ is rigid and that $M$ satisfies $(S_n)$ for some $n\geq0$. If $\depth_R(N)-n\leq\CI_R(M)$, then for all $i>0$ in the range $\depth_R(N)-n\leq i\leq\CI_R(M)$, we have
$\Ext^i_R(M,N)\neq0$.
\end{thm}
\begin{proof}
If $\Ext^i_R(M,N)=0$ for some $\depth_R(N)-n\leq i\leq\CI_R(M)$, then $\Ext^1_R(\Omega^{i-1}M,N)\cong\Ext^i_R(M,N)=0$. Note that
$\Omega^{i-1}M$ satisfies $(S_{n+i-1})$. Now by Theorem \ref{t12}(i), we have $\CI_R(\Omega^{i-1}M)=\G-dim_R(\Omega^{i-1}M)=0$. Therefore $\CI_R(M)<i$ by \cite[Lemma 1.9]{AGP}, which is a contradiction.
\end{proof}
Let $R$ be a hypersurface and let $M$ and $N$ be $R$--modules such that $\len_R(N)<\infty$. It is well-known that if $\Ext^i_R(M,N)=0$ for some $i>\CI_R(M)$, then $\Ext^n_R(M,N)=0$ for all $n>\CI_R(M)$ (see for example \cite[Corollary 3.5]{B1}). In special cases, we can remove the condition that $i>\CI_R(M)$.
\begin{cor}\label{cor4}
Let $(R,\fm)$ be a local hypersurface ring such that $\widehat{R} = S/(f)$ where $(S, \fn)$ is a
complete unramified regular local ring and $f$ is a regular element of $S$
contained in ${\fn}^2$. Let $M$ and $N$ be nonzero $R$-modules such that $\len_R(N)<\infty$.
If $\Ext^n_R(M,N)=0$ for some $n\geq1$, then  the following statements hold true.
\begin{enumerate}[(i)]
\item{$\CI_R(M)=\sup\{i\mid\Ext^i_R(M,N)\neq0\}<n$.}
\item{either $\pd_R(M)<\infty$ or $\pd_R(N)<\infty$.}
\end{enumerate}
\end{cor}
\begin{proof}
First note that $N$ is rigid by \cite[Theorem 2.4]{HW}. It follows from Theorem \ref{cor7} that $\CI_R(M)=\sup\{i\mid\Ext^i_R(M,N)\neq0\}<n$. As $\Ext^i_R(M,N)=0$ for all $i\gg0$, either $\pd_R(M)<\infty$ or $\pd_R(N)<\infty$ by Theorem \ref{th1}.
\end{proof}
As an application of Theorem \ref{t12}, we have the following result.
\begin{cor}\label{cor12}
Let $R$ be an admissible hypersurface and let $M$ and $N$ be nonzero $R$--modules such that $\cx_R(N)=1$. Assume that the minimal free resolution of $N$ is eventually periodic of period one and that $M$ satisfies $(S_n)$ for some $n\geq0$.
Then the following statements hold true.
\begin{enumerate}[(i)]
\item{If $\depth_R(N)-n\leq\CI_R(M)$, then for all $i>0$ in the range $\depth_R(N)-n\leq i\leq\CI_R(M)$, we have $\Ext^i_R(M,N)\neq0$.}
\item{If $\Ext^i_R(M,N)=0$ for some positive integer $i$ such that $i\geq\depth_R(N)-n$, then $\pd_R(M)<i$.}
\item{If $\Ext^i_R(M,N)=0$ for all $i$, $1\leq i\leq\max\{1,\depth_R(N)-2\}$, then $M^*$ is free.}
\end{enumerate}
\end{cor}
\begin{proof}
Note that $N$ is rigid by \cite[Corollary 5.6]{Da}. Now the first assertion is clear by Theorem \ref{cor10}.

(ii). By Theorem \ref{cor7}, $\Ext^j_R(M,N)=0$ for all $j\geq i$. Therefore, $\pd_R(M)<\infty$ by Theorem \ref{th1} and so
$\pd_R(M)<i$.

(iii). Note that $N$ has maximal complexity. Therefore, the assertion is clear by Corollary \ref{co}.
\end{proof}
Let $R$ be an admissible hypersurface with isolated singularity of dimension $d>1$. By \cite[Theorem 3.4]{Da}, every $R$--module of dimension less than or equal one is rigid. As an immediate consequence of Theorem \ref{cor7}, we have the following result.
\begin{cor}
Let $R$ be an admissible hypersurface with isolated singularity of dimension $d>1$ and let $M$ and $N$ be nonzero $R$--modules such that $\dim_R(N)\leq1$. If $\Ext^n_R(M,N)=0$ for some $n>0$, then $\CI_R(M)=\sup\{i\mid\Ext^i_R(M,N)\neq0\}<n$. Moreover, either $\pd_R(M)<\infty$ or $\pd_R(N)<\infty$.
\end{cor}
In the dimension 2 case, we have the following result.
\begin{prop}\label{th5}
Let $R$ be an admissible hypersurface of dimension $2$. Assume further that $R$ is normal. Let $M$ and $N$ be nonzero $R$--modules such that
$\depth_R(N)\leq\depth_R(M)+1$. If $\Ext^1_R(M,N)=0$, then $\CI_R(M)=0$ and $\Ext^i_R(M,N)=0$ for all $i>0$. Moreover, either $M$ is free or $N$ has finite projective dimension.
\end{prop}
\begin{proof}
First note that $N$ is rigid by \cite[Corollary 3.6]{Da}.
If $\depth_R(N)\leq1$, then the assertion is clear by Theorem \ref{cor7}. Now let $N$ be maximal Cohen-Macaulay. Then $\depth_R(M)>0$ and so
\begin{equation}\tag{\ref{th5}.1}
\CI_R(M)=\sup\{i\mid\Ext^i_R(M,R)\neq0\}=2-\depth_R(M)\leq1.
\end{equation}
By Theorem \ref{a1}, $\Tor_1^R(\mathcal{T}_2M,N)=0$. As $N$ is rigid, $\Tor_i^R(\mathcal{T}_2M,N)=0$ for all $i>0$. It follows from Theorem \ref{a1} again that $\Ext^1_R(M,R)=0$ and so $M$ is maximal Cohen-Macaulay by (\ref{th5}.1). Now it is easy to see that $\Tr M\approx\Omega\mathcal{T}_2M$ and so $\Tor_i^R(\Tr M,N)=0$ for all $i>0$. Therefore, $\Ext^i_R(M,N)=0$ for all $i>0$ by Theorem \ref{th3} and so either $M$ is free or $N$ has finite projective dimension by Theorem \ref{th1}.
\end{proof}
\section{Vanishing of Ext over complete intersection rings}
Let $R$ be a local complete intersection ring of codimension $c$ and let $M$ and $N$ be $R$--modules.
In \cite{Mu}, Murthy proved that if $\Tor_n^R(M,N)=\Tor_{n+1}^R(M,N)=\cdots=\Tor_{n+c}^R(M,N)=0$ for some $n>0$, then $\Tor_i^R(M,N)=0$ for all $i\geq n$. It is easy to see that a similar statement is not true in general, with Tor replaced by Ext. In the following, we prove
a similar result for Ext with an extra hypothesis. The following result is a generalization of \cite[Corollary]{Jot1}.
\begin{thm}\label{t4}
Let $R$ be a local complete intersection ring of codimension $c$ and let $M$ and $N$ be nonzero $R$--modules. Assume $n$ is a positive integer. If $\Ext^{i}_R(M,N)=0$, for all $i$, $n\leq i\leq n+c$ and $\depth_R(N)\leq n+c$, then $\CI_R(M)=\sup\{i\mid\Ext^i_R(M,N)\neq0\}<n$.
\end{thm}
\begin{proof}
Without lose of generality we may assume that $R$ is complete.
We have $R=Q/(x)$ with $Q$ a complete regular local ring and  $x$ an $Q$-sequence of length $c$ contained in the square of
the maximal ideal of $Q$. We argue by induction on $c$. If $c=0$, then $R$ is a regular local ring and so $\pd_R(M)=\sup\{i\mid\Ext^i_R(M,N)\neq0\}<n$ by \cite[Corollary 1]{Jot2}.
For $c>0$, set $S=Q/(x_1,\ldots, x_{c-1})$. Therefore $R\cong S/(x_c)$. Note that $\depth_R(N)=\depth_S(N)$.

The change of rings spectral sequence (see \cite[Theorem 11.66]{R})
$$\Ext^p_R(M,\Ext^q_S(R,N))\underset{p}{\Rightarrow}\Ext^{p+q}_S(M,N)$$
degenerates into a long exact sequence
$$\cdots\rightarrow\Ext^i_R(M,N)\rightarrow\Ext^i_S(M,N)\rightarrow\Ext^{i-1}_R(M,N)\rightarrow\Ext^{i+1}_R(M,N)\rightarrow\cdots.$$
It follows that $\Ext^i_S(M,N)=0$ for all $i$, $n+1\leq i\leq n+c$, and so by induction hypothesis we conclude that $\CI_S(M)=\sup\{i\mid\Ext^i_S(M,N)\neq0\}<n+1$. Therefore, $\Ext^{i-1}_R(M,N)\cong\Ext^{i+1}_R(M,N)$ for all $i>n$. As $c>0$, it is clear that $\Ext^i_R(M,N)=0$ for all $i\geq n$ and so $\CI_R(M)=\sup\{i\mid\Ext^i_R(M,N)\neq0\}<n$ by Theorem \ref{th2}.
\end{proof}
In special cases, one can improve the Theorem \ref{t4} slightly. The following is a generalization of corollary \ref{cor4}.
\begin{prop}\label{t7}
Let $(R,\fm)$ be a local ring such that $\widehat{R} = S/(f)$ where $(S, \fn)$ is a
complete unramified regular local ring and $f = f_1, f_2, \ldots, f_c$ is a regular sequence of $S$
contained in ${\fn}^2$. Assume that $n\geq0$ is an integer and that $M$ and $N$ are nonzero finite $R$-modules such that $\len_R(N)<\infty$.
If $\Ext^i_R(M,N)=0$ for all $i$, $n+1\leq i\leq n+c$, then $\CI_R(M)=\sup\{i\mid\Ext^i_R(M,N)\neq0\}\leq n$.
\end{prop}
\begin{proof}
Without lose of generality we may assume that $R$ is complete and $R=S/(f)$ where $(S, \fn)$ is a
complete unramified regular local ring and $f = f_1, f_2, \ldots, f_c$ is a regular sequence of $S$ contained in ${\fn}^2$.
We argue by induction on $c$. If $c=1$, then the assertion holds by Corollary \ref{cor4}. For $c>1$, set $Q=S/(f_1,\ldots, f_{c-1})$.
Therefore, $R\cong Q/(f_c)$. Note that $\len_Q(N)<\infty$.
The change of rings spectral sequence
$$\Ext^p_R(M,\Ext^q_Q(R,N))\underset{p}{\Rightarrow}\Ext^{p+q}_Q(M,N)$$
degenerates into a long exact sequence
$$\cdots\rightarrow\Ext^i_R(M,N)\rightarrow\Ext^i_Q(M,N)\rightarrow\Ext^{i-1}_R (M,N)\rightarrow\Ext^{i+1}_R(M,N)\rightarrow\cdots.$$
It follows that $\Ext^i_Q(M,N)= 0$ for all $i$, $n+2\leq i \leq n+c$, and so by
induction hypothesis we conclude that $\CI_Q(M)\leq n+1$ and $\Ext^i_Q(M,N)=0$ for
all $i > n+1$. Therefore, $\Ext^{i-1}_R(M,N)\cong\Ext^{i+1}_R(M,N)$ for all $i > n+1$. As $c > 1$, it is
clear that $\Ext^i_R(M,N) = 0$ for all $i > n$ and so $\CI_R(M)=\sup\{i\mid\Ext^i_R(M,N)\neq0\}\leq n$ by Theorem \ref{th2}.
\end{proof}
As an application of Theorem \ref{t4}, we can generalize \cite[Corollary 1]{Jot2} as follows.
\begin{cor}\label{cor11}
Let $R$ be a local complete intersection ring of codimension $c$ and let $M$ and $N$ be nonzero $R$--modules. Assume that $n>0$ and $t\geq0$ are integers and that the following conditions hold.
\begin{enumerate}[(i)]
\item $\Ext^{i}_R(M,N)=0$ for all $i$, $n\leq i\leq n+c$.
\item $M$ satisfies $(S_t)$.
\item $\depth_R(N)\leq n+c+t$.
\end{enumerate}
Then $\CI_R(M)=\sup\{i\mid\Ext^i_R(M,N)\neq0\}<n$.
\end{cor}
\begin{proof}
We argue by induction on $t$. If $t=0$, then the assertion is clear by Theorem \ref{t4}. Now suppose that $t>0$ and consider the universal pushforward of $M$,
\begin{equation}\tag{\ref{cor11}.1}
0\rightarrow M\rightarrow F\rightarrow M_1\rightarrow0,
\end{equation}
where $F$ is free. It is easy to see that $M_1$ satisfies $(S_{t-1})$. From the exact sequence (\ref{cor11}.1), it is clear that
\begin{equation}\tag{\ref{cor11}.2}
\Ext^i_R(M,N)\cong\Ext^{i+1}_R(M_1,N) \text{ for all } i>0.
\end{equation}
Therefore, $\Ext^i_R(M_1,N)=0$ for all $i$, $n+1\leq i\leq n+c+1$. By induction hypothesis, we conclude that $\Ext^i_R(M_1,N)=0$ for all $i>n$. By (\ref{cor11}.2), $\Ext^i_R(M,N)=0$ for all $i\geq n$ and so $\CI_R(M)=\sup\{i\mid\Ext^i_R(M,N)\neq0\}<n$ by Theorem \ref{th2}.
\end{proof}
{\bf Acknowledgements}.
I would like to thank  Olgur Celikbas and my thesis adviser Mohammad Taghi Dibaei for
valuable suggestions and comments on this paper.
\bibliographystyle{amsplain}

\end{document}